%% file: SplittingRAAGs-arxiv.tex
\documentclass[11pt]{amsart}
\usepackage{amsfonts, amsmath, amssymb, amscd, amsthm, graphicx,enumerate}
\usepackage{epsfig, color}
\usepackage{hyperref, stmaryrd}

\makeatletter
\def\blfootnote{\xdef\@thefnmark{}\@footnotetext}
\makeatother

\newtheorem{thm}{Theorem}[section]

\newtheorem{lem}[thm]{Lemma}
\newtheorem{prop}[thm]{Proposition}

\newtheorem{defn}[thm]{Definition}
\newtheorem{rem}[thm]{Remark}

\newtheorem{thmA}{Theorem}

\newfont{\eufm}{eufm10}

\renewcommand{\phi}{\varphi}

\newcommand{\Z}{\mathbb Z}
\newcommand{\R}{\mathbb R}

\def\star {\mathrm{star}}
\def\link {\mathrm{link}}

\begin{document}

\title{Abelian splittings of right-angled Artin groups}
\author{Daniel Groves}
\address{Department of Mathematics, Statistics and Computer Science, University of Illinois at Chicago}
\email{groves@math.uic.edu}

\author{Michael Hull}
\address{Department of Mathematics, Statistics and Computer Science, University of Illinois at Chicago}
\email{mbhull@uic.edu}

\date{}
\maketitle

\begin{abstract}
We characterize when (and how) a Right-Angled Artin group splits nontrivially over an abelian subgroup.
\end{abstract}

Given a graph $\Gamma$, the associated \emph{Right-Angled Artin group} (abbreviated RAAG) $A(\Gamma)$ is the group given by the presentation
\[
A(\Gamma)=\langle V(\Gamma)\;|\; [u, v]=1 \text{ whenever } (u, v)\in E(\Gamma)\rangle.
\]

If $\Gamma$ has no edges, then $A(\Gamma)$ is a free group while if $\Gamma$ is a clique (i.e. a complete graph), then $A(\Gamma)$ is a free abelian group. Hence is natural to view the class of RAAGs as a class of groups which interpolates between free and free abelian groups. Additionally, RAAGs have become important in geometry and topology through the work of Haglund and Wise \cite{HW08,HaglundWise12} on (virtually) special cube complexes, which in turn played a key role in Agol's solution \cite{agol:virtualhaken} of the Virtual Haken Conjecture for 3-manifolds (and even more so in Agol's resolution of the Virtual Fibering Conjecture, using the RFRS condition). 

One particularly nice property of RAAGs is that many of the algebraic properties of $A(\Gamma)$ can be easily computed from the graph $\Gamma$. For example, $A(\Gamma)$ splits as a non-trivial free product if and only if $\Gamma$ is not connected and $\Gamma$ splits as a non-trivial direct product if and only if $\Gamma$ is a join. In a similar vein (though this is harder), Clay \cite{Clay} recently showed that for a connected $\Gamma$ which is not a single edge, $A(\Gamma)$ splits over a cyclic subgroup if and only $\Gamma$ has a cut-vertex. The goal of this paper is to generalize Clay's result to give a characterization of which RAAGs split nontrivially over an abelian subgroup in terms of the defining graph.

One reason to be interested in splittings of RAAGs over abelian subgroups is to study the (outer) automorphism group of a RAAG.  
There has been a lot of recent interest in automorphisms of RAAGs (for example, see \cite{CharneyVogtmann}, and also Vogtmann's lectures at MSJ-SI \cite{MSJ-SI}), partly in analogy with the study of $\mathrm{Out}(F_n)$.  One approach to studying $\mathrm{Out}(F_n)$ that proved quite fruitful is Sela's \cite{SelaNielsen}, where limiting actions on $\mathbb{R}$-trees and splittings of $F_n$ were studied.  The splittings in this case were over infinite cyclic subgroups.  For RAAGs, the `extension graph' built by Kim and Koberda \cite{KimKoberda} (see also Kim's talk at MSJ-SI \cite{MSJ-SI}) might be used to provide limiting $\mathbb{R}$-trees equipped with an action of a RAAG in an analogous way to Sela's approach.  Thus, it is of considerable interest to understand the ways in which a RAAG can split over an abelian subgroup.  We intend to investigate limiting actions of RAAGs on $\R$-trees arising in this way in a future paper.

As with free and cyclic splittings, there are a few cases where an abelian splitting can be directly seen in the graph $\Gamma$. If $\Gamma$ is disconnected then it splits as a free product and if $\Gamma$ is a complete graph on $n$ vertices then $A(\Gamma) \cong \Z^n$, which splits nontrivially over $\Z^{n-1}$.  Slightly more interestingly, if $K$ is a separating clique, then let $\Gamma \smallsetminus K = \Gamma_0 \sqcup \Gamma_1$ where $\Gamma_0$ and $\Gamma_1$ are each nonempty and share no vertices.  Let $\bar{\Gamma}_0 = \Gamma_0 \cup K$ and $\bar{\Gamma}_1 = \Gamma_1 \cup K$.  The defining presentation for $A(\Gamma)$ exhibits it as $A(\bar{\Gamma}_0) \ast_{A(K)} A(\bar{\Gamma}_1)$.  Note that $A(K)$ is free abelian, and properly contained in $A(\bar{\Gamma}_i)$, so this is a nontrivial abelian splitting of $A(\Gamma)$.

Our main theorem is that the only RAAGs which split over an abelian subgroup are those which have one of the obvious splittings mentioned above.

\begin{thmA} \label{mainthm}
 Suppose that $\Gamma$ is a finite simplicial graph.  Then the associated Right-Angled Artin group $A(\Gamma)$ splits nontrivially over an abelian group if and only if one of the following occurs:
\begin{enumerate}
 \item $\Gamma$ is disconnected;
 \item $\Gamma$ is a complete graph; or
 \item $\Gamma$ contains a separating clique.
\end{enumerate}
\end{thmA}

\begin{rem}
The proof of Theorem \ref{mainthm} below gives a simpler proof of \cite[Theorem A]{Clay}.  In particular, if one reads this proof with the extra assumption that $A(\Gamma)$ splits nontrivially over $\Z$ then this proof will yield a cut vertex in $\Gamma$.  Clay indicated to us that he has a proof of Theorem \ref{mainthm} similar to the one we give below.
\end{rem}

In Section \ref{sec:JSJ} we then describe a `JSJ' decomposition which encodes (to a certain extent) all of the ways that a RAAG can split over an abelian subgroup.

The first author gratefully acknowledges the support of the NSF, and thanks the Mathematical Society of Japan for the opportunity to speak at the 7th Seasonal Institute of the MSJ.

\section{Proof of Theorem \ref{mainthm}}

\begin{lem} \label{l:fix intersect}
 Suppose that a group $G$ acts on a tree $T$ and that $g,h \in G$ are commuting elements which both act elliptically on $T$.  Then $\mathrm{Fix}(g) \cap \mathrm{Fix}(h) \ne \emptyset$.
\end{lem}
\begin{proof}
Since $g$ and $h$ commute $g$ leaves the subtree $\mathrm{Fix}(h)$ invariant.  However, it also fixes $\mathrm{Fix}(g)$, which means that if these trees are disjoint $g$ must fix (pointwise) the unique shortest path $p$ between $\mathrm{Fix}(g)$ and $\mathrm{Fix}(h)$.  This implies that in fact $p$ must be contained in $\mathrm{Fix}(g)$, which means that the fixed trees are not in fact disjoint.  
\end{proof}

Recall that given a vertex $v$ in a graph $\Gamma$, $\link(v)$ is the subgraph induced by $\{u\in V(\Gamma)\;|\; (v, u)\in E(\Gamma)\}$ and $\star(v)$ is the subgraph induced by $\link(v)\cup\{v\}$.

\begin{lem}\label{l:hyp vert}
If $A(\Gamma)$ acts on a tree with abelian edge stabilizers and $v\in V(\Gamma)$ is hyperbolic, then $\star(v)$ is a clique and $\{u\in \link(v)\;|\; u \text{ is elliptic}\}$ separates $v$ and $\Gamma\setminus \link(v)$.
\end{lem}

\begin{proof}
Let $u, w\in \link(v)$. Since $u$ and $w$ commute with $v$, they must preserve the axis of $v$ setwise. Hence we can find integers $n_1, n_2, m_1, m_2$ with $n_1,m_1 \ne 0$ so that each of $g = u^{n_1}v^{n_2}$ and $h = w^{m_1}v^{m_2}$ fix the axis of $v$ pointwise, in particular the subgroup $\langle g, h\rangle$ stabilizes an edge of $T$. Thus $\langle g, h\rangle$ is abelian, which can only happen if $u$ and $w$ commute. This shows that $\star(v)$ is a clique. By the same proof, if $v^\prime\in \star(v)$ is hyperbolic, then $\star(v^\prime)$ is a clique and hence $\star(v^\prime)=\star(v)$. Thus, if $\Gamma\setminus \star(v)\neq\emptyset$, then $\Gamma\setminus\{u\in \link(v)\;|\; u \text{ is elliptic}\}$ is disconnected.
\end{proof}

\begin{proof} (Theorem \ref{mainthm})
Suppose that $\Gamma$ is connected and not a complete graph, and suppose that $A(\Gamma)$ acts on a tree $T$ with no global fixed points and abelian edge stabilizers. We use the action of $A(\Gamma)$ on $T$ to find a separating clique in $\Gamma$.

First, suppose that some $v\in V(\Gamma)$ acts hyperbolically on $T$. Then by Lemma \ref{l:hyp vert}, $\star(v)$ is a clique and hence not all of $\Gamma$. Then $\link(v)$ is a clique which separates $v$ from $\Gamma\setminus \star(v)$, so $\Gamma$ contains a separating clique.

Thus, we may suppose that each vertex of $\Gamma$ acts elliptically on $T$. Since the action has no global fixed point, it follows from Helly's theorem that for some $x, y\in V(\Gamma)$, $\mathrm{Fix}(x)\cap \mathrm{Fix}(y)=\emptyset$. Let $q$ be the shortest path between $\mathrm{Fix}(x)$ and $\mathrm{Fix}(y)$, and let $p$ be a point in the interior of an edge on $q$ which is not in $\mathrm{Fix}(x)$ or $\mathrm{Fix}(y)$. We define a map $F : \Gamma \to T$ as follows: first choose $a_x \in \mathrm{Fix}(x)$ and $a_y \in \mathrm{Fix}(y)$ and define $F(x) = a_x$ and $F(y) = a_y$. Now for vertices $u$ such that $p \in \mathrm{Fix}(u)$, define $F(u) = p$. For any other vertex $z$ of $\Gamma$, choose some $a_z \in \mathrm{Fix}(z)$ and define $F(z) = a_z$. For an edge $e$ of $\Gamma$ between vertices $u$ and $v$, let $F$ continuously map $e$ onto the (possibly degenerate) geodesic between $F(u)$ and $F(v)$. This defines a continuous map $F : \Gamma \to T$.
Lemma \ref{l:fix intersect} implies (for vertices $u$ and $v$ which are adjacent in $\Gamma$) that $\mathrm{Fix}(u) \cap \mathrm{Fix}(v) \ne \emptyset$, hence $F(e)\subseteq \mathrm{Fix}(u)\cup \mathrm{Fix}(v)$.  Consequently we have $F(\Gamma)\subseteq \bigcup_{v\in V(\Gamma)}\mathrm{Fix}(v)$.

Since $\Gamma$ is connected and $F$ is continuous, the image $F(\Gamma)$ is connected.  This implies that $p \in F(\Gamma)\subseteq \bigcup_{v\in V(\Gamma)}\mathrm{Fix}(v)$. Hence $F^{-1}(p)$ is a non-empty subgraph of $\Gamma$, and since $p$ separates $F(\Gamma)$ (since $T$ is a tree), $F^{-1}(p)$ separates $\Gamma$. The subgroup generated by the vertices of $F^{-1}(p)$ all fix the edge containing $p$, hence this subgroup is abelian and therefore $F^{-1}(p)$ is a separating clique.
\end{proof}

\section{Vertex-elliptic abelian JSJ decomposition}\label{sec:JSJ}

Suppose that $\Gamma$ is a finite connected graph.  In this section, we build a JSJ decomposition for $A(\Gamma)$, a decomposition which in some sense encodes all possible abelian splittings of $A(\Gamma)$. However, we first restrict our attention to those splittings in which each vertex of $\Gamma$ acts elliptically. The reason for this restriction is that actions with hyperbolic vertices all arise through simplicial actions of $\mathbb Z^n$ on a line, and such an action can be easily modified to make $\mathbb Z^n$ act elliptically. Indeed, suppose $A(\Gamma)$ acts on a tree $T$ with abelian edge stabilizers such that some $v\in V(\Gamma)$ acts hyperbolically. By Lemma \ref{l:hyp vert}, $\star(v)$ is a clique and hence $A(\star(v))\cong \mathbb Z^n$ for some $n\geq 1$. Let $\star(v)=\star_h(v)\sqcup\star_e(v)$, where each vertex of $\star_h(v)$ acts hyperbolically on $T$ and each vertex of $\star_e(v)$ acts elliptically. Let $\Gamma^\prime=\Gamma\setminus\star_h(v)$. By Lemma \ref{l:hyp vert}, $A(\Gamma)$ has the splitting
\[
A(\Gamma)=A(\Gamma^\prime)\ast_{A(\star_e(v))}A(\star(v)).
\]
Furthermore this splitting is compatible with the splitting of $A(\Gamma^\prime)$ induced by its action on $T$ since $A(\star_e(v))$ acts elliptically on $T$. Hence we get a new splitting of $A(\Gamma)$ such that $v$ is elliptic and the induced splitting of $A(\Gamma^\prime)$ is the same as the splitting induced by its action on $T$. Repeating this procedure, any abelian splitting of $A(\Gamma)$ can be replaced by a splitting in which all vertices act elliptically.   We call such a splitting a {\em vertex-elliptic abelian} splitting.

If $\Gamma$ is a complete graph, then there are no nontrivial vertex-elliptic abelian splittings of $A(\Gamma)$.  Therefore, we henceforth suppose that $\Gamma$ is both connected and not a complete graph.

One consequence of considering only vertex-elliptic actions is that HNN-extensions cannot arise: the corresponding graph of groups is a tree.
\begin{prop}
Let $A(\Gamma)$ act on a tree such that each vertex of $\Gamma$ acts elliptically. Then the corresponding graph of groups is a tree.
\end{prop}
\begin{proof}
Let $\mathcal G$ be the corresponding graph of groups decomposition of $A(\Gamma)$, and let $G$ be the underlying graph of $\mathcal G$. Then there is a natural surjective map $A(\Gamma)\to\pi_1(G)$ which maps the elliptic elements of $A(\Gamma)$ to the identity. But $A(\Gamma)$ is generated by elliptic elements, so $\pi_1(G)=\{1\}$ which means that $G$ is a tree.
\end{proof}

The following lemma is an immediate consequence of Lemma \ref{l:fix intersect} and Helly's theorem.

\begin{lem} \label{l:clique elliptic}
 Suppose $A(\Gamma)$ acts on a tree $T$ such that each vertex of $\Gamma$ acts elliptically. Then for any clique $K$ in $\Gamma$, $A(K)$ acts elliptically on $T$.
\end{lem}

We will use the language of Guirardel-Levitt \cite{GuiLevJSJI} in order to describe the vertex-elliptic abelian JSJ decomposition of $A(\Gamma)$. Given $A(\Gamma)$, let $\mathcal A$ be the set of abelian subgroups of $A(\Gamma)$ and $\mathcal H=\{\langle v\rangle\;|\; v\in V(\Gamma)\}$. If $A(\Gamma)$ acts  on a tree $T$ such that each edge stabilizer is abelian and each vetex of $\Gamma$ is elliptic, then $T$ is called an $(\mathcal A, \mathcal H)$-tree. An $(\mathcal A, \mathcal H)$-tree $T$ is called \emph{universally elliptic} if the edge stabilizers of $T$ are elliptic in every other $(\mathcal A, \mathcal H)$-tree. Also, given trees $T$ and $T^\prime$, we say $T$ \emph{dominates} $T^\prime$ if every vertex stabilizer of $T$ is elliptic with respect to $T^\prime$.

\begin{defn}
A graph of groups decomposition $\mathcal G$ of $A(\Gamma)$ is called a \emph{vertex-elliptic abelian JSJ-decomposition of $A(\Gamma)$} if the corresponding Bass-Serre tree $T$ is an $(\mathcal A, \mathcal H)$-tree  which is universally elliptic and which dominates every other universally elliptic $(\mathcal A, \mathcal H)$-tree.
\end{defn}

\begin{figure}
\centering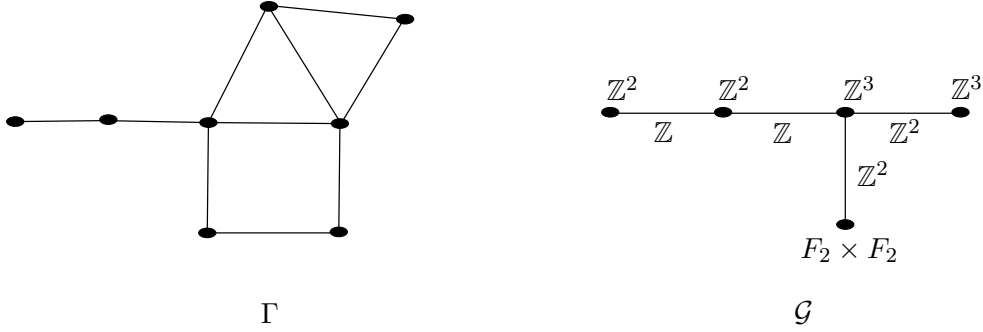\\
  \caption{A graph $\Gamma$ with the vertex-elliptic abelian JSJ-decomposition of $A(\Gamma)$.}\label{fig1}
\end{figure}

We will now build a vertex-elliptic abelian JSJ decomposition for the group $A(\Gamma)$.  If $\Gamma$ has no separating cliques then this JSJ decomposition is trivial. Otherwise, suppose that the minimal size of a separating clique in $\Gamma$ is $k$.  Let $K^1, \ldots , K^n$ be the collection of all separating cliques of size $k$.

Each $K^i$ defines a vertex-elliptic abelian splitting of $A(\Gamma)$ as follows. The underlying graph of the splitting is a tree with a valence-one vertex for each component $\Gamma^{i,j}$ of $\Gamma \smallsetminus K^i$, all adjacent to single central vertex. The vertex corresponding to $\Gamma^{i,j}$ has associated vertex group $A(\Gamma^{i,j} \cup K^i)$, while the central vertex has vertex group $A(K^i)$.  All of the edge groups correspond to $A(K^i)$ in the natural way.  In case there are only two components of $\Gamma \smallsetminus K^i$, this is a reducible two-edge splitting, which can be transformed into a one-edge splitting in the obvious way.

By Lemma \ref{l:clique elliptic} each $A(K^j)$ is elliptic in the splitting defined by $K^i$.  Therefore, the splittings with edge groups $A(K^i)$ can be refined into a single (vertex-elliptic abelian) splitting $\Lambda$ of $A(\Gamma)$.  The vertex groups of $\Lambda$ naturally correspond to connected subgraphs of $\Gamma$.
For each such subgraph $\Gamma^i$ we can repeat this procedure with any separating cliques of minimal size that it has (they are necessarily of size larger than $k$). Since we are only considering vertex-elliptic splittings,  Lemma \ref{l:clique elliptic} implies that any such splitting of $A(\Gamma^i)$ can be extended to a vertex-elliptic abelian splitting of $A(\Gamma)$ in an obvious way.

Eventually, we find a vertex-elliptic abelian splitting $\mathcal G$ of $A(\Gamma)$ so that the subgraphs associated to the vertex groups are connected and contain no separating cliques.
See Figure \ref{fig1} for an example where one edge adjacent to each reducible valence two vertex has been contracted.

\begin{thm}
$\mathcal G$ is a vertex-elliptic abelian JSJ-decomposition of $A(\Gamma)$.
\end{thm}

\begin{proof}
Let $T$ be the Bass-Serre tree corresponding to $\mathcal G$. Note that, up to conjugation, each edge stabilizer in $T$ is of the form $A(K)$, where $K$ is a clique inside of $\Gamma$. Suppose $A(\Gamma)$ acts on another $(\mathcal A, \mathcal H)$-tree $T^\prime$. Then by Lemma \ref{l:clique elliptic}, for every clique $K$ of $\Gamma$, $A(K)$ acts elliptically on $T^\prime$, hence $T$ is universally elliptic. 

Next we show that for any $(\mathcal A,\mathcal H)$-tree $T^\prime$, $T$ dominates $T^\prime$. Up to conjugation, each vertex stabilizer of $T$ is of the form $A(\Gamma^\prime)$, where $\Gamma^\prime$ is a connected subgraph of $\Gamma$  which has no separating cliques. If $\Gamma^\prime$ is a clique, then $A(\Gamma^\prime)$ acts elliptically on $T^\prime$ by the previous paragraph. If $\Gamma^\prime$ is not a clique, then $A(\Gamma^\prime)$ must act elliptically on $T^\prime$ otherwise this action would violate Theorem \ref{mainthm}. 
\end{proof}

\small
\bibliographystyle{alpha} 

\def\cprime{$'$}

\end{document}

%% file: RAAGSJSJ.pdf_tex
\begingroup%
  \makeatletter%
  \providecommand\color[2][]{%
    \errmessage{(Inkscape) Color is used for the text in Inkscape, but the package 'color.sty' is not loaded}%
    \renewcommand\color[2][]{}%
  }%
  \providecommand\transparent[1]{%
    \errmessage{(Inkscape) Transparency is used (non-zero) for the text in Inkscape, but the package 'transparent.sty' is not loaded}%
    \renewcommand\transparent[1]{}%
  }%
  \providecommand\rotatebox[2]{#2}%
  \ifx\svgwidth\undefined%
    \setlength{\unitlength}{423.22036133bp}%
    \ifx\svgscale\undefined%
      \relax%
    \else%
      \setlength{\unitlength}{\unitlength * \real{\svgscale}}%
    \fi%
  \else%
    \setlength{\unitlength}{\svgwidth}%
  \fi%
  \global\let\svgwidth\undefined%
  \global\let\svgscale\undefined%
  \makeatother%
  \begin{picture}(1,0.28858586)%
    \put(0,0){\includegraphics[width=\unitlength]{RAAGSJSJ.pdf}}%
    \put(0.22673217,0.00209766){\color[rgb]{0,0,0}\makebox(0,0)[lb]{\smash{$\Gamma$}}}%
    \put(0.53473788,0.19999603){\color[rgb]{0,0,0}\makebox(0,0)[lb]{\smash{$\mathbb Z^2$}}}%
    \put(0.63421288,0.20022155){\color[rgb]{0,0,0}\makebox(0,0)[lb]{\smash{$\mathbb Z^2$}}}%
    \put(0.74568811,0.20007906){\color[rgb]{0,0,0}\makebox(0,0)[lb]{\smash{$\mathbb Z^3$}}}%
    \put(0.84145237,0.20087369){\color[rgb]{0,0,0}\makebox(0,0)[lb]{\smash{$\mathbb Z^3$}}}%
    \put(0.70156576,0.00387464){\color[rgb]{0,0,0}\makebox(0,0)[lb]{\smash{$\mathcal G$}}}%
    \put(0.57662854,0.16328652){\color[rgb]{0,0,0}\makebox(0,0)[lb]{\smash{$\mathbb Z$}}}%
    \put(0.68273553,0.16219343){\color[rgb]{0,0,0}\makebox(0,0)[lb]{\smash{$\mathbb Z$}}}%
    \put(0.7863658,0.16198182){\color[rgb]{0,0,0}\makebox(0,0)[lb]{\smash{$\mathbb Z^2$}}}%
    \put(0.70737128,0.05951032){\color[rgb]{0,0,0}\makebox(0,0)[lb]{\smash{$F_2\times F_2$}}}%
    \put(0.75768847,0.12407159){\color[rgb]{0,0,0}\makebox(0,0)[lb]{\smash{$\mathbb Z^2$}}}%
  \end{picture}%
\endgroup%